\documentclass[12pt,a4paper]{article}
\usepackage[utf8]{inputenc}
\usepackage[T1]{fontenc}
\usepackage{amsmath,amssymb}
\usepackage{amsfonts}
\usepackage{amssymb,amsthm}
\usepackage{graphicx}
\usepackage{float}
\usepackage{subcaption}
\usepackage[left=2.00cm, right=2.00cm, top=2.00cm, bottom=3.00cm]{geometry}
\usepackage{cite}

\numberwithin{equation}{section}

\newtheorem{thm}{Theorem}[section]
\newtheorem{prop}[thm]{Proposition}
\newtheorem{cor}[thm]{Corollary}
\newtheorem{lem}[thm]{Lemma}
\newtheorem{ex}[thm]{Example}
\newtheorem{defn}[thm]{Definition}
\newtheorem{rem}[thm]{Remark}
\allowdisplaybreaks

\author{}
\title{Pedal curves of hyperbolic frontals and their singularities}
\author{\\\
        {\bf O. O\u{g}ulcan Tuncer}$^{1,2}$ \thanks{E-mail address:
        otuncer@hacettepe.edu.tr}
    {\,, \bf \.{I}smail G\"{o}k}$^2$\thanks{E-mail address:
       igok@science.ankara.edu.tr}
 \medskip\\
   \\$^1$Department of Mathematics, Hacettepe University,
                    \\ 06800 Beytepe, Ankara, Turkey
  \medskip\\$^2$Department of Mathematics, Ankara University,
                   \\ 06100 Be\c{s}evler, Ankara, Turkey}
\begin{document}

\maketitle
\begin{abstract}
This paper introduces pedal curves of spacelike frontals in the hyperbolic 2-space. We mainly investigate the singularities of these hyperbolic pedal curves of spacelike frontals for non-singular and singular dual curve germs. We then show that for non-singular dual curve germs, the singularities of a pedal curve are dependent on the singularities of the first hyperbolic Legendrian curvature germ and the location of the pedal point, while for singular dual curve germs, they depend upon the singularities of both hyperbolic Legendrian curvature germs and also the location of the pedal point. We provide several examples with figures.

%for non-singular dual curve germs, one can determine singularity types of pedal curves based on the singularities of the first hyperbolic Legendrian curvature germs and locations of pedal points. On the other hand, for singular dual curve germs, singularity types of pedal curves depend upon singularities of both of the hyperbolic Legendrian curvature germs and also locations of pedal points. Then, after introducing orthotomics of spacelike frontals, we present some relationships of such curves with hyperbolic pedal curves. Finally, we give an application to light patterns generated by reflected rays in the hyperbolic 2-space.
\end{abstract}
\section{Introduction}
Generating new curves based on some rules from a given curve is one of the widely studied problems in Differential Geometry. Among these curves, pedal curves have an importance. A pedal curve of a regular curve in the Euclidean plane is the locus of the feet of the perpendiculars from a fixed point (the pedal point) to the tangent lines along the curve \cite{gray,zwikker}. This definition yields pedal coordinates of a point on the curve with respect to the curve and the pedal point. Based on these coordinates, we are able to write the pedal equation of a given curve. Furthermore, pedal coordinates are practical for solving specific force problems in the classical and celestial mechanics. In \cite{blaschke} the author states that the trajectory of a particle under central and Lorentz-like forces can be converted to pedal coordinates at once without need of solving any differential equation. Then the methods developed in \cite{blaschke} is used to solve dark Kepler problem. Pedal coordinates are also used to investigate orbits of a free double linkage \cite{blaschke2} and certain variational problems \cite{blaschke3}.

We can give a more general definition of pedal curves: A pedal curve of a regular curve is defined as the locus of the nearest point in the geodesic, which is tangent to the curve at a point, from a given point. This definition allows us to give parametric representations of pedal curves by using the orthogonal projection and the Frenet frame along the curve. But if the curve is not regular at some points, then the pedal curve cannot be defined as above since the Frenet frame is not well-defined along the curve. In \cite{fukunaga} Fukunaga and Takahashi examined Legendre curves in the unit tangent bundle of the Euclidean plane and introduced a moving frame called the Legendrian Frenet frame along the curve. This frame is well-defined even at singular points of the curve. Using the Legendrian Frenet frame, Fukunaga and Takahashi \cite{fukunaga2,fukunaga3,fukunaga4} provided the definitions of the evolute and involute of a curve, which may have singularities. After this study, spherical fronts in the Euclidean 2-sphere were defined, and evolutes of spherical fronts were examined by Yu et al. \cite{yu}. Li and Pei \cite{li} considered pedal curves of spherical fronts by using the definition of pedal curves of spherical regular curves \cite{nishimura,nishimura2}. Moreover Chen and Takahashi \cite{chen} introduced spacelike and timelike frontals in the hyperbolic and de Sitter 2-spaces, and provided Legendrian Frenet frames along these spacelike and timelike frontals. Using these frames, they defined evolutes and parallels of timelike and spacelike fronts. Recently, many studies on frontals and framed curves (in 3D setting) have been proposed (\cite{bekar,honda,honda2,li2,li3,li4,li5,li6,li7,tuncer}).

%Similar to the above general definition of pedal curves, an orthotomic of a curve $\mathbf{r}$ relative to a point $Q$ can be defined as the set of the reflections of $Q$ about the planes in which the geodesics tangent to the $\mathbf{r}$ lie for all $s\in I$. The evolute of the orthotomic is called the catacaustic. From point of the view of optics, the catacaustic represents the light patterns generated by reflected rays. See for more information to \cite{arnold,brugibgib,georgiou}.  

The purpose of this paper is to introduce hyperbolic pedal curves of frontals in the hyperbolic 2-space, and investigate singularities of these pedal curves.

This paper is organized as follows. In Section 2 we give some preliminaries on Minkowski geometry and spacelike frontals in hyperbolic 2-space. We also introduce pedal curves of regular curves in hyperbolic 2-space. In Section 3 the definitions of hyperbolic pedal curves of spacelike frontals are provided, and some geometric properties of these curves are given. In Section 4 we give a complete classification of singularities of hyperbolic pedal curves based on the dual curve germs being non-singular or singular. In Section 5 we conclude about our results.

\section{Preliminaries}
The Lorentz Minkowski 3-space $\mathbb{R}^3_1$ is the real vector space with a pseudo scalar product given by 
\begin{equation*}
	\langle u,w\rangle=-u_1w_1+u_2w_2+u_3w_3, 
\end{equation*}
where $u=(u_1,u_2,u_3),\, w=(w_1,w_2,w_3)\in\mathbb{R}^3$. \\ 
\indent The vectors in $\mathbb{R}^3_1$ are classified depending on this pseudo scalar product. Consider a non-zero vector $u=(u_1,u_2,u_3)\in\mathbb{R}^3_1$. The vector $u$ is called a spacelike, a timelike or a lightlike (null) vector if $\langle u,u\rangle>0$, $\langle u,u\rangle<0$ or $\langle u,u\rangle=0$, respectively. The pseudo-norm of the vector $u$ is given by $\|u\|=\sqrt{\lvert\langle u,u\rangle\rvert}$. For two arbitrary vectors $u=(u_1,u_2,u_3)$ and $w=(w_1,w_2,w_3)$, the Lorentzian vector product is defined by
\begin{equation*}
	u\wedge w = 
	\begin{vmatrix}
		-e_1& e_2 & e_3\\
		u_1& u_2 & u_3 \\
		w_1& w_2 & w_3 
	\end{vmatrix} 
	= (-u_2w_3+u_3w_2, u_3w_1-u_1w_3, -u_2w_1+u_1w_2),
\end{equation*}
where the set $\{e_1,e_2,e_3\}$ is the canonical basis of $\mathbb{R}^3_1$.\\ \indent
In Lorentz Minkowski 3-space, curves are classified depending on their tangent vectors. A curve is said to be spacelike, timelike, or lightlike (null) if the tangent vector of the curve is spacelike, timelike, or lightlike (null), respectively.
%\indent Given a vector $\mathbf{u}$ and a real number $c$, the plane with pseudo-normal $\mathbf{u}$ is defined by
%\begin{equation*}
%	P(\mathbf{u},c)=\{\mathbf{x}\in\mathbb{R}_1^3\,\vert\,\langle\mathbf{x},\mathbf{u}\rangle=c \}.
%\end{equation*}
%The plane $P(\mathbf{u},c)$ is classified depending on its pseudo-normal $\mathbf{u}$. If $\mathbf{u}$ is a spacelike, a timelike or a lightlike vector, then the plane $P(\mathbf{u},c)$ is said to be a timelike, a spacelike or a lightlike plane, respectively. 

Now we recall the two dimensional pseudo spheres in the Lorentz Minkowski 3-space. The hyperbolic 2-space, de Sitter 2-space, and lightlike cone at the origin are respectively defined by
\begin{align*}
	&\mathcal{H}^2=\{\mathbf{u}\in\mathbb{R}^3_1\,\vert\,\langle \mathbf{u},\mathbf{u}\rangle=-1 \},\\
	&d\mathcal{S}^2=\{\mathbf{u}\in\mathbb{R}^3_1\,\vert\,\langle \mathbf{u},\mathbf{u}\rangle=1 \},\\
	&LC^*=\{\mathbf{u}\in\mathbb{R}^3_1\backslash\{\mathbf{0}\} \,\vert\,\langle \mathbf{u},\mathbf{u}\rangle=0 \}.
\end{align*} 
%Let us take a curve obtained by the intersection of $\mathcal{H}^2$ (or, $d\mathcal{S}^2$) with the plane $P(\mathbf{u},c)$:
%$$ \mathcal{H}^2\cap P(\mathbf{u},c)\,(\textrm{or},\,d\mathcal{S}^2\cap P(\mathbf{u},c)).$$
%Then, if $\mathbf{u}$ is a timelike, a spacelike, or a lightlike curve, then the intersection curve is called hyperbolic (or de Sitter) ellipse, hyperbolic (or de Sitter) parabola or hyperbolic (or de Sitter) hyperbola, respectively. 
For further details about Lorentz Minkowski space see \cite{o'neil, lopez}. 
\subsection{Pedal curves of regular curves in hyperbolic 2-space}
Let $\mathbf{r}_h:I\to \mathcal{H}^2$ be a regular curve, that is $\|{\mathbf{r}}'_h(s)\|\neq0$ for all $s\in I$, where ${\mathbf{r}}'_h(s)$ denotes the derivative of $\mathbf{r}_h$ with respect to $s$. Since $\mathbf{r}_h$ is a regular curve, the unit spacelike tangent vector field $T_h(s)={\mathbf{r}}'_h(s)/\|{\mathbf{r}}'_h(s)\|$ to the curve is well-defined. Hence taking a unit spacelike vector $N_h=\mathbf{r}_h\wedge T_h$, we obtain a pseudo orthonormal frame $\{\mathbf{r}_h,T_h,N_h\}$ along $\mathbf{r}_h$ called the hyperbolic Frenet frame. Then the hyperbolic Frenet-Serret formulas of this frame are given by
\begin{equation}
	\begin{pmatrix}
		{\mathbf{r}}'_h(s)\\
		{T}'_h(s)\\
		{N}'_h(s)
	\end{pmatrix}=\|{\mathbf{r}}'_h(s)\| \begin{pmatrix}
		0 & 1 & 0 \\ 1 & 0 & \kappa_h(s) \\ 0 & -\kappa_h(s) & 0
	\end{pmatrix}\begin{pmatrix}
		\mathbf{r}_h(s)\\ T_h(s) \\ N_h(s)
	\end{pmatrix},
\end{equation}
where $\kappa_h(s)=\dfrac{\det(\mathbf{r}_h(s),{\mathbf{r}}'_h(s),{\mathbf{r}}''_h(s))}{\| {\mathbf{r}}'_h(s)\|^3}$ is called the hyperbolic geodesic curvature. 

The pedal curve of $\mathbf{r}_h$ with respect to the point  $Q\in \mathcal{H}^2$ is given by
\begin{equation}\label{regpedal}
	Ped_Q(\mathbf{r}_h)(s)=\dfrac{1}{\sqrt{1+\langle Q, N_h(s)\rangle ^2}}(Q-\langle Q,N_h(s)\rangle N_h(s)),
\end{equation}
It is easy to see that $Ped_Q(\mathbf{r}_h)(s)$ is located in $\mathcal{H}^2$.
\subsection{Spacelike frontals in the hyperbolic 2-space}
In this section we briefly review spacelike frontals in the hyperbolic 2-space. See \cite{chen} for more information.
Suppose that $\mathbf{r}_h:I\to \mathcal{H}^2$ is a spacelike curve at regular points. If there exists a smooth map $v_h:I\to d\mathcal{S}^2$ such that $(\mathbf{r}_h,v_h):I\to \Delta_1$ satisfies $(\mathbf{r}_h(s),v_h(s))^*\theta=0$ for any $s\in I$, then the curve $\mathbf{r}_h$ (resp. the pair $(\mathbf{r}_h,v_h)$) is called a spacelike frontal (resp. a spacelike Legendre curve) in hyperbolic 2-space (resp. in $\Delta_1$), where
\begin{equation*}
	\Delta_1=\{(\mathbf{u},\mathbf{w})\,\vert\,\langle\mathbf{u},\mathbf{w}\rangle=0\}\subset \mathcal{H}^2 \times d\mathcal{S}^2
\end{equation*}
is a 3-dimensional manifold, and $\theta$ is a canonical contact 1-form on $\Delta_1$. The condition $(\mathbf{r}_h(s), v_h(s))^*\theta=0$ means $\langle {\mathbf{r}}'_h(s), v_h(s)\rangle=0$ for all $s\in I$. In addition, if $(\mathbf{r}_h,v_h)$ is an immersion, then this curve is called a spacelike Legendrian immersion in $\Delta_1$, and $\mathbf{r}_h$ is called a spacelike front in the hyperbolic 2-space. The set $\{\mathbf{r}_h(s), v_h(s),\mu_h(s)=\mathbf{r}_h(s)\wedge v_h(s) \}$ is a pseudo orthonormal frame along $\mathbf{r}_h$ called the hyperbolic Legendrian Frenet frame. Moreover this frame is well-defined even at singular points of $\mathbf{r}_h$. The hyperbolic Legendrian Frenet-Serret formulas are given by
\begin{equation}\label{eq2.3}
	\begin{pmatrix}
		{\mathbf{r}}'_h(s)\\
		{v}'_h(s)\\
		{\mu}'_h(s)
	\end{pmatrix}=\begin{pmatrix}
		0 & 0 & \ell_h(s) \\ 0& 0 & m_h(s) \\ \ell_h(s) & -m_h(s) & 0
	\end{pmatrix}\begin{pmatrix}
		\mathbf{r}_h(s)\\ v_h(s) \\ \mu_h(s)
	\end{pmatrix},
\end{equation}
where, the pair $( \ell_h(s)=\langle {\mathbf{r}}'_h(s), \mu_h(s)\rangle, m_h(s)=\langle {v}'_h(s), \mu_h(s)\rangle)$ is called the spacelike hyperbolic Legendrian curvature of $(\mathbf{r}_h,v_h)$. Furthermore we call $v_h$ the dual curve of $\mathbf{r}_h$.
\section{Pedal curves of spacelike frontals in the hyperbolic plane}
Consider a spacelike Legendre curve $(\mathbf{r}_h,v_h)$ with spacelike hyperbolic Legendre curvature $(\ell_h,m_h)$. The pedal curve of $\mathbf{r}_h$ with respect to the pedal point $Q\in \mathcal{H}^2_+$ is the locus of the nearest point in the geodesic $G_{v_h(s)}$ from $Q$ that is tangent to the vector $\mu_h$ at a point (See Fig. \ref{pedal}).
\begin{figure}[H] 
	\centering
	\includegraphics[width=0.5\textwidth]{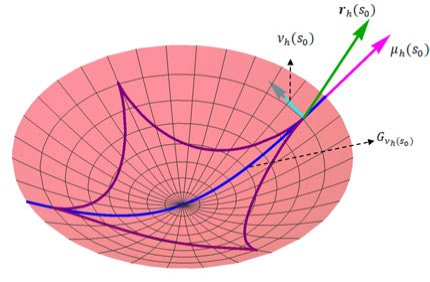}
	\caption{Geometry of a pedal curve in the hyperbolic 2-space} \label{pedal}
\end{figure}
By this definition, $Q$ and $\mathbf{r}_h$ must lie on the same part of the hyperbolic 2-space. The pedal curve $\mathcal{P}ed_Q(\mathbf{r}_h)(s):I\to \mathcal{H}^2$ of the spacelike frontal $\mathbf{r}_h(s)$ relative to the point $Q\in \mathcal{H}^2$ is parametrized by
\begin{equation}\label{eq3.2}
	\mathcal{P}ed_Q(\mathbf{r}_h)(s)=\dfrac{1}{\sqrt{1+\langle Q, v_h(s)\rangle ^2}}(Q-\langle Q,v_h(s)\rangle v_h(s)),
\end{equation}
Note that $Q\in\mathcal{H}^2$ and $v_h(s_0)\in d\mathcal{S}^2$. So $Q=v_h(s_0)$ is not allowed in \eqref{eq3.2}.
%\begin{remark}
%	It seems like that if $Q=v_h(s_0)$, then the equation of the pedal curve defined above vanishes at that point. But, we cannot consider this case since $Q\in \mathcal{H}^2$ while $v_h(s_0)\in d\mathcal{S}^2$.
%\end{remark}
%It is easy to see that the definitions given above can be also used spacelike fronts in hyperbolic 2-space. From now on, we assume $(\ell_h,m_h)\neq(0,0)$, that is $\mathbf{r}_h$ is a spacelike front.
\begin{prop}\label{prop3.2}
	Let $\mathbf{r}_h:I\to \mathcal{H}^2$ be a regular curve, and let $Q$ be a point in $\mathcal{H}^2$. Then the pedal curve of the regular spacelike curve coincides with the pedal curve of the spacelike front.
\end{prop}
\begin{proof}
	Without loss of generality, by taking $v_h(s)=N_h(s)$ we obtain a spacelike Legendre immersion $(\mathbf{r}_h, N_h)$ with the spacelike hyperbolic Legendre curvature $(-\|{\mathbf{r}}'_h\|,\|{\mathbf{r}}'_h\|\kappa_h)$. From \eqref{regpedal} and \eqref{eq3.2}
	\begin{align*}
		Ped_Q(\mathbf{r}_h)(s)&=\dfrac{1}{\sqrt{1+\langle Q, N_h(s)\rangle ^2}}(Q-\langle Q,N_h(s)\rangle N_h(s))\\
		&=\dfrac{1}{\sqrt{1+\langle Q, v_h(s)\rangle ^2}}(Q-\langle Q,v_h(s)\rangle v_h(s))=	\mathcal{P}ed_Q(\mathbf{r}_h)(s).
	\end{align*}
\end{proof}
\begin{prop}\label{prop3.3}
	Let $(\mathbf{r}_h,v_h):I\to\Delta_1$ be a spacelike Legendre curve with spacelike hyperbolic Legendre curvature $(\ell_h,m_h)$, and let $Q$ be a point in $\mathcal{H}^2$. Then the pedal curve $\mathcal{P}ed_Q(\mathbf{r}_h)$ of the spacelike frontal $\mathbf{r}_h$ is not dependent of the parametrization of $(\mathbf{r}_h,v_h)$.
\end{prop}
\begin{proof}
	Suppose that $(\mathbf{r}_h,v_h):I\to \Delta_1\subset \mathcal{H}^2\times d\mathcal{S}^2$ and $(\tilde{\mathbf{r}}_h,\tilde{v}_h):\tilde{I}\to \Delta_1\subset \mathcal{H}^2\times d\mathcal{S}^2$ are parametrically equivalent by means of a (positive) change of parameter $s:\tilde{I}\to I$. Then we have $(\tilde{\mathbf{r}}_h(\xi),\tilde{v}_h(\xi))=(\mathbf{r}_h(s(\xi)),v_h(s(\xi)))$, and so we obtain
	\begin{align*}
		\mathcal{P}ed_Q(\tilde{\mathbf{r}}_h)(\xi)&=\dfrac{1}{\sqrt{1+\langle Q, \tilde{v}_h(\xi)\rangle ^2}}(Q-\langle Q,\tilde{v}_h(\xi)\rangle \tilde{v}_h(\xi))\\
		&=\dfrac{1}{\sqrt{1+\langle Q, v_h(s(\xi))\rangle ^2}}(Q-\langle Q,v_h(s(\xi))\rangle v_h(s(\xi)))=	\mathcal{P}ed_Q(\mathbf{r}_h)(s(\xi)).
	\end{align*}	
\end{proof}

\begin{thm}\label{thm3.3}
	Consider a spacelike Legendre curve $(\mathbf{r}_h,v_h)$ with the spacelike hyperbolic Legendre curvature $(\ell_h,m_h)$. The pedal curve $\mathcal{P}ed_Q(\mathbf{r}_h)$ of the spacelike frontal $\mathbf{r}_h$ relative to $Q\in \mathcal{H}^2$ has a singular point at $s_0$ if and only if $m_h(s_0)=0$ or $Q=\mathbf{r}_h(s_0)$. 
\end{thm}
\begin{proof}
	If we differentiate Eq. \eqref{eq3.2}, then using hyperbolic Legendrian Frenet-Serret type formulas we obtain
	\begin{align}\label{eq3.4}
		\nonumber	{\mathcal{P}}ed'_Q(\mathbf{r}_h)(s)=&-\dfrac{m_h(s)}{\sqrt{1+\langle Q, v_h(s)\rangle^2}}\bigg(\langle Q,\mu_h(s)\rangle\,v_h(s)+\langle Q,v_h(s)\rangle\,\mu_h(s)\bigg) \medskip\\
		&-m_h(s)\dfrac{\langle Q,v_h(s)\rangle\langle Q,\mu_h(s)\rangle}{(1+\langle Q, v_h(s)\rangle^2)^{3/2}}\bigg(-\langle Q,\mathbf{r}_h(s)\rangle\,\mathbf{r}_h(s)+\langle Q,\mu_h(s)\rangle\,\mu_h(s)\bigg).
	\end{align}
	Then it is easy to see that $s_0\in I$ is a singular point of $\mathcal{P}ed_Q(\mathbf{r}_h)$ if and only if $m_h(s_0)=0$ or $Q=\mathbf{r}_h(s_0)$ since $\{\mathbf{r}_h(s), v_h(s),\mu_h(s) \}$ is a pseudo orthonormal frame. 
\end{proof}

\begin{cor}
	Suppose that $(\mathbf{r}_h,v_h):I\to\Delta_1$ is a spacelike Legendre curve with the spacelike hyperbolic Legendre curvature $(\ell_h,m_h)$. If $s_0$ is a singular point of $v_h$, then the pedal curve $\mathcal{P}ed_Q(\mathbf{r}_h)$ is also singular at $s_0$.
\end{cor}
\begin{proof}
	Since $s_0$ is a singular point of $v_h$, we have $m_h(s_0)=0$. Then using \eqref{eq3.4} we conclude the proof.
\end{proof}

\begin{thm}\label{thm3.6}
	Let $(\mathbf{r}_h,v_h)$ be a spacelike Legendre curve with spacelike hyperbolic Legendre curvature $(\ell_h,m_h)$, and let $Q$ be a point in $\mathcal{H}^2-\mathbf{r}_h(I)$. Then the pedal curve $\mathcal{P}ed_Q(\mathbf{r}_h)$ of $\mathbf{r}_h$ with respect to $Q$ is a spacelike frontal, that is, $(\mathcal{P}ed_Q(\mathbf{r}_h),\breve{v}_h)$ is a spacelike Legendre curve with spacelike hyperbolic Legendre curvature $(\breve{\ell}_h,\breve{m}_h)$, where
	\begin{equation}\label{eq112}
		\begin{array}{ll}
			\breve{v}_h&=\dfrac{\langle Q,\mu_h\rangle^2\mathbf{r}_h+\langle Q,\mathbf{r}_h\rangle\langle Q,v_h\rangle v_h-\langle Q,\mathbf{r}_h\rangle\langle Q,\mu_h\rangle \mu_h }{\sqrt{\langle Q,\mu_h\rangle^2(1+\langle Q,v_h\rangle^2)+\langle Q,\mathbf{r}_h\rangle^2\langle Q,v_h\rangle^2 }},\medskip \\
			\breve{\mu}_h&=\dfrac{\langle Q,\mathbf{r}_h\rangle^2\langle Q,v_h\rangle \mu_h+\langle Q,\mu_h\rangle(\langle Q,\mathbf{r}_h\rangle^2-\langle Q,\mu_h\rangle^2)v_h-\langle Q,\mathbf{r}_h\rangle\langle Q,v_h\rangle\langle Q,\mu_h\rangle\mathbf{r}_h}{\sqrt{1+\langle Q,v_h\rangle^2}\sqrt{\langle Q,\mu_h\rangle^2(1+\langle Q,v_h\rangle^2)+\langle Q,\mathbf{r}_h\rangle^2\langle Q,v_h\rangle^2}} \medskip \\
			\breve{\ell}_h&=\dfrac{m_h}{1+\langle Q,v_h\rangle^2}\sqrt{\langle Q,\mu_h\rangle^2(1+\langle Q,v_h\rangle^2)+\langle Q,\mathbf{r}_h\rangle^2\langle Q,v_h\rangle^2 }, \medskip \\
			\breve{m}_h&=\langle {\breve{v}}'_h(s), \breve{\mu}_h(s)\rangle.
		\end{array}
	\end{equation} 
\end{thm}
\begin{proof}
	It is enough to show that $(\mathcal{P}ed_Q(\mathbf{r}_h),\breve{v}_h)$ satisfies the conditions for being a spacelike Legendre curve.
	
	The first condition is $ \langle\mathcal{P}ed_Q(\mathbf{r}_h),\breve{v}_h\rangle=0$. By using Eq. \eqref{eq3.2} and $\breve{v}_h$ defined in \eqref{eq112}, we get
	{\small\begin{align*}
			\langle\mathcal{P}ed_Q(\mathbf{r}_h),\breve{v}_h\rangle&= \dfrac{\langle Q,\mu_h\rangle^2\langle Q,\mathbf{r}_h\rangle+\langle Q,\mathbf{r}_h\rangle\langle Q,v_h\rangle^2-\langle Q,\mathbf{r}_h\rangle\langle Q,\mu_h\rangle^2-\langle Q,\mathbf{r}_h\rangle\langle Q,v_h\rangle^2}{\sqrt{1+\langle Q, v_h(s)\rangle^2}\sqrt{\langle Q,\mu_h\rangle^2(1+\langle Q,v_h\rangle^2)+\langle Q,\mathbf{r}_h\rangle^2\langle Q,v_h\rangle^2 }} \\&=0.
	\end{align*}}
	The second condition is $\langle {\mathcal{P}}'ed_Q(\mathbf{r}_h),\breve{v}_h\rangle=0$. Using Eq. \eqref{eq3.4} and $\breve{v}_h$ defined in \eqref{eq112} yields
	{\small \begin{align*}
			\langle {\mathcal{P}}ed'_Q(\mathbf{r}_h),\breve{v}_h\rangle=&A(s) \bigg( \dfrac{1}{1+\langle Q, v_h(s)\rangle^2}(-\langle Q, \mu_h(s)\rangle^3\langle Q, \mathbf{r}_h(s)\rangle\langle Q, v_h(s)\rangle\\
			&+\langle Q, \mathbf{r}_h(s)\rangle\langle Q, \mu_h(s)\rangle^3\langle Q, v_h(s)\rangle)\\
			&-\langle Q, \mu_h(s)\rangle\langle Q, \mathbf{r}_h(s)\rangle\langle Q, v_h(s)\rangle+\langle Q, \mathbf{r}_h(s)\rangle\langle Q, \mu_h(s)\rangle\langle Q, v_h(s)\rangle \bigg)\\
			=&0,
	\end{align*}}
	where 
	\[A(s)=\dfrac{m_h(s)}{\sqrt{1+\langle Q, v_h(s)\rangle^2}\sqrt{\langle Q,\mu_h\rangle^2(1+\langle Q,v_h\rangle^2)+\langle Q,\mathbf{r}_h\rangle^2\langle Q,v_h\rangle^2 }}.\]
	Then these two conditions show that $(\mathcal{P}ed_Q(\mathbf{r}_h),\breve{v}_h)$ is a spacelike Legendre curve. Since the set $\{\mathbf{r}_h,v_h,\mu_h\}$ constructs a pseudo orthonormal basis we can write $Q-\langle Q,v_h\rangle v_h=-\langle Q,\mathbf{r}_h\rangle \mathbf{r}_h+\langle Q,\mu_h\rangle \mu_h$. By considering this fact one can calculate the Lorentzian vector product $\mathcal{P}ed_Q(\mathbf{r}_h)\wedge \breve{v}_h$ which gives $\breve{\mu}_h$. Using these relations, one can obtain the Legendre curvatures of $\mathcal{P}ed_Q(\mathbf{r}_h)$.
\end{proof}

\begin{ex}\label{ex3.7}
	For the hyperbolic astroid given by \\$\mathbf{r}_h(s)=(\sqrt{1+\cos^6s+\sin^6s},\cos^3s,\sin^3s)$, we get
	\[{\mathbf{r}}'_h(s)=(\dfrac{-3\cos^5s\sin s+3\sin^5s\cos s}{\sqrt{1+\cos^6s+\sin^6s}}, -3\cos^2s\sin s, 3\sin^2s\cos s). \]
	Take $v_h:[0,2\pi)\to d\mathcal{S}^2$ 
	{\small\[v_h(s)=\dfrac{1}{\sqrt{1+\sin^2s\cos^2s}}\bigg(\sin s\cos s\sqrt{1+\cos^6s+\sin^6s},\sin s(1+\cos^4s),\cos s(1+\sin^4s)\bigg). \]}
	Then $(\mathbf{r}_h,v_h):[0,2\pi)\to \mathcal{H}^2\times d\mathcal{S}^2$ is a spacelike Legendre immersion. Moreover by the Lorentzian vector product $\mathbf{r}_h\wedge v_h$ we immediately get
	\[\mu_h(s)=\dfrac{\sqrt{1+\cos^6s+\sin^6s}}{\sqrt{1+\sin^2s\cos^2s}}\bigg(\dfrac{\cos^4s-\sin^4 s}{\sqrt{1+\cos^6s+\sin^6s}},\cos s,-\sin s\bigg). \]
	Choose the pedal point $Q_1=(1,0,0)$. Then the hyperbolic pedal of the spacelike front $\mathbf{r}_h$ with respect to $Q_1$ is obtained as (See Fig. \ref{fig1}):
	\begin{align*}
		\mathcal{P}ed_{Q_1}(\mathbf{r}_h)=&\dfrac{1}{\sqrt{1+\cos^2s \sin^2 s}\sqrt{1+\cos^2s \sin^2s \left(2+\cos^6s+\sin^6s\right)}}\\
		&\times\bigg(1+\cos^2s \sin^2s \left(2+\cos^6s+\sin^6s\right),\\
		&\cos s
		\left(1+\cos^4s\right) \sin^2s \sqrt{1+\cos^6s+\sin^6s},\\
		&
		\cos^2s \sin s \left(1+\sin^4s\right) \sqrt{1+\cos^6s+\sin^6s}\bigg).
	\end{align*}
	\begin{figure}[H]
		\centering
		\includegraphics[width=0.6\textwidth]{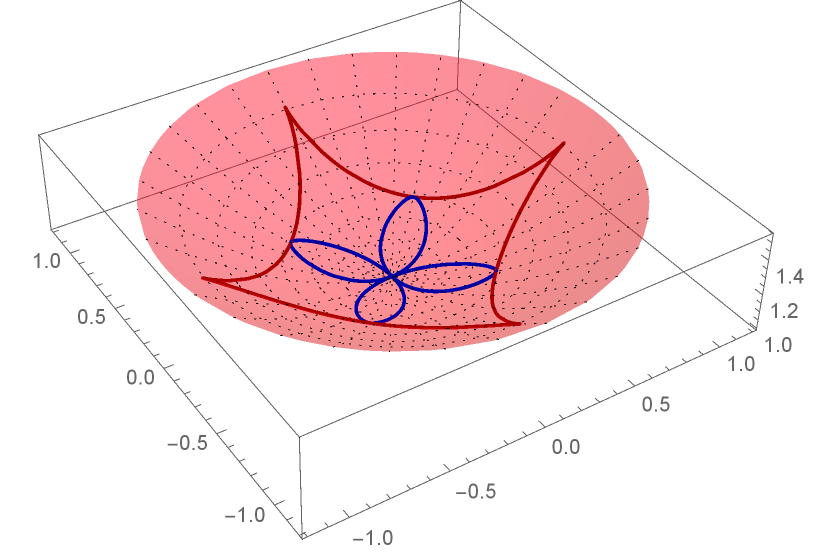}
		\caption{The hyperbolic astroid (red) and its hyperbolic pedal curve (blue) with respect to $Q_1=(1,0,0)$.} \label{fig1}
	\end{figure}
	Now we choose the pedal point $Q_2=\dfrac{1}{2}(\sqrt{5},\dfrac{1}{\sqrt{2}},\dfrac{1}{\sqrt{2}})=\mathbf{r}_h(\pi/4)$. Then the hyperbolic pedal of the spacelike front $\mathbf{r}_h$ with respect to $Q_2$ can be obtained similarly. (See Fig. \ref{fig3}).
	\begin{figure}[H]
		\centering
		\includegraphics[width=0.6\textwidth]{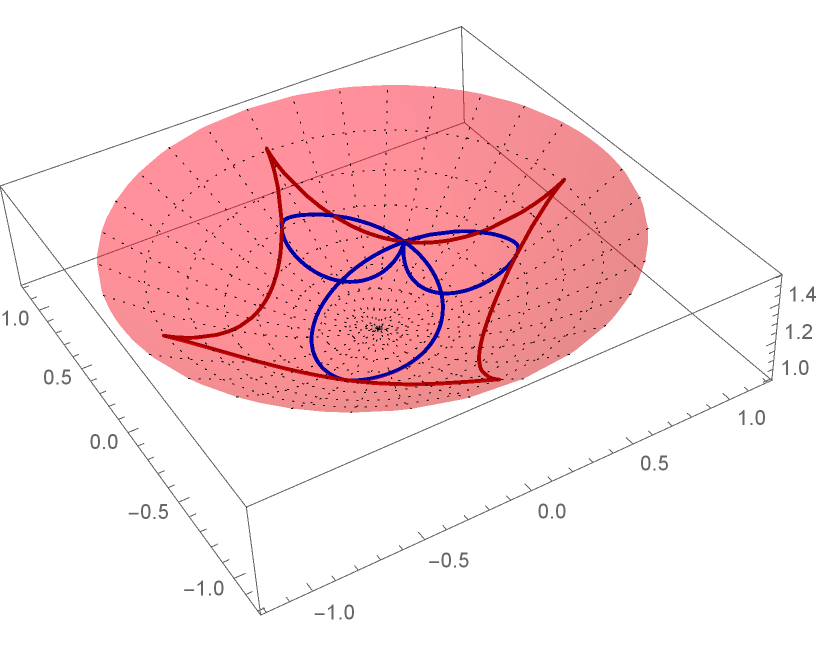}
		\caption{The hyperbolic astroid (red) and its hyperbolic pedal curve (blue) with respect to $Q_2=\dfrac{1}{2}(\sqrt{5},\dfrac{1}{\sqrt{2}},\dfrac{1}{\sqrt{2}})$.} \label{fig3}
	\end{figure}
\end{ex}
\section{Classification of singularities of hyperbolic pedal curves}
Now we give the complete classification of singularities of hyperbolic pedal curves. We present singularity types of pedal curves for non-singular and singular dual curve germs. For the rest of this paper, we shall assume $Q,\mathbf{r}_h\in \mathcal{H}^2_+$, where $\mathcal{H}^2_+$ represents the upper hyperbolic $2$-space.

Let $Q$ be a point of $\mathcal{H}^2_+$. Consider the $C^\infty$ map 
\begin{align*}
	\psi_Q: d\mathcal{S}^2&\to \mathcal{H}^2_+ \\
	x&\mapsto \psi_Q(x)=\dfrac{1}{\sqrt{1+\langle Q,x\rangle^2}}(Q-\langle Q,x\rangle x).
\end{align*}
Define $\tilde{H}_Q=\{y\in \mathcal{H}^2_+\,\vert\,\langle Q,y\rangle <0 \}$. Since $\langle\psi_Q(x),Q\rangle <0$ for all $x\in d\mathcal{S}^2$ we have $X_Q:=\psi_Q(d\mathcal{S}^2)\subset\tilde{H}_Q$. Let $d\mathbb{S}^2:=d\mathcal{S}^2/\{\pm Id\}$ be the model space (or projective quotient) of $d\mathcal{S}^2$. By construction, $d\mathbb{S}^2$ is a subset of the projective space $\mathbb{R}P^2$. Indeed, topologically $d\mathbb{S}^2$ can be defined by
\begin{equation*}
	d\mathbb{S}^2=P\{ x\in\mathbb{R}^3_1\,\vert\,\langle x,x\rangle > 0 \}.
\end{equation*}
Consider the canonical projection $f:d\mathcal{S}^2\to d\mathbb{S}^2$.  Since $\psi_Q(x)=\psi_Q(-x)$, the map $\psi_Q$ induces $\tilde{\psi}_Q:d\mathbb{S}^2\to X_Q$. Thus it follows from \eqref{eq3.2} that 
\begin{equation*}
	\mathcal{P}ed_Q(\mathbf{r}_h)(s) =\tilde{\psi}_Q\circ f \circ v_h(s).
\end{equation*}
Suppose that $b:B\to\mathbb{R}^2_1$ is the blow up of $\mathbb{R}^2_1$ centered at the origin, where $B=\{(x_1,x_2)\times[y_1:y_2]\in\mathbb{R}^2_1\times\mathbb{R}P^1\,\vert\,x_1y_2=x_2y_1\}$. We give the following lemma.
\begin{lem}\label{lemma4.1}
	Let $Q\in\mathcal{H}^2_+$. Then there exist $C^\infty$ diffeomorphisms $h_s:d\mathbb{S}^2\to B$ and $h_t:X_Q\to\mathbb{R}^2_1$ such that $h_t\circ \tilde{\psi}_Q\equiv b\circ h_s$.
\end{lem}
\begin{proof}
	Without loss of generality assume that $Q=(1,0,0)$. For
	\begin{align*}
		U_1=\{ (x_1,x_2)\times[y_1:y_2]\in \mathbb{R}^2_1\times \mathbb{R}P^1\,\vert\,x_1y_2=x_2y_1,\quad x_1\neq 0 \}, \\
		U_2=\{ (x_1,x_2)\times[y_1:y_2]\in \mathbb{R}^2_1\times \mathbb{R}P^1\,\vert\,x_1y_2=x_2y_1,\quad x_2\neq 0 \},
	\end{align*}
	and 
	\begin{align*}
		\varphi_1:U_1\to\mathbb{R}^2_1;\quad (x_1,x_2)\times[y_1:y_2]\mapsto(u_1,u_2)=(x_1,\dfrac{y_2}{y_1}), \\
		\varphi_2:U_2\to\mathbb{R}^2_1;\quad (x_1,x_2)\times[y_1:y_2]\mapsto(u_1',u_2')=(\dfrac{y_1}{y_2},x_2),
	\end{align*}
	it is well-known that the set $\{ (U_1,\varphi_1),(U_2,\varphi_2) \}$ is the standard atlas for $B$ and we have 
	\begin{align*}
		b\circ \varphi_1^{-1}(u_1,u_2)=(u_1,u_1u_2),\\
		b\circ \varphi_2^{-1}(u_1,u_2)=(u_1u_2,u_2).
	\end{align*}
	Define the sets 
	$U_{Q,1}=\{ f(x_1,x_2,x_3)\,\vert\,x_2\neq 0 \}$, 
	$U_{Q,2}=\{ f(x_1,x_2,x_3)\,\vert\,x_3\neq 0 \},$
	and the maps
	\begin{align*}
		\varphi_{Q,1}:U_{Q,1}\to\mathbb{R}^2_1;\quad\varphi_{Q,1}(f(x_1,x_2,x_3))=(\tanh(\lambda)x_2,\dfrac{x_3}{x_2}), \\
		\varphi_{Q,2}:U_{Q,2}\to\mathbb{R}^2_1;\quad\varphi_{Q,2}(f(x_1,x_2,x_3))=(\dfrac{x_2}{x_3},\tanh(\lambda)x_3),
	\end{align*}
	where $\lambda=\sinh^{-1}(x_1)\in\mathbb{R}$. Then the following equality holds.
	\begin{equation*}
		\varphi_{Q,j}\circ\varphi_{Q,i}^{-1} \equiv \varphi_j\circ\varphi_i^{-1},\quad i,j\in\{1,2\}.
	\end{equation*}
	Therefore the set $\{(U_{Q,1},\varphi_{Q,1}),(U_{Q,2},\varphi_{Q,2})\}$ is an atlas for $d\mathbb{S}^2$.
	
	Now we express the map $\tilde{\psi}_Q$ by the coordinates $(u_1,u_2,u_3)$. From the assumption $Q=(1,0,0)$, for $x=(\sinh\lambda,x_2,x_3)$ we have
	$\psi_Q(x)=(\cosh\lambda, x_2 \tanh\lambda, x_3 \tanh\lambda).$
	Then we get 
	\begin{align}\label{eq4.1}
		q\circ \tilde{\psi}_Q\circ \varphi_{Q,1}^{-1}(u_1,u_2)=(u_1,u_2u_1),\\
		q\circ \tilde{\psi}_Q\circ \varphi_{Q,2}^{-1}(u_1,u_2)=(u_1u_2,u_2), \nonumber
	\end{align}
	where $q:\mathbb{R}^3_1\to\mathbb{R}^2_1;$ $(x,y,z)\mapsto (y,z)$ is the canonical projection. The restriction $q\vert_{d\mathbb{S}^2}:d\mathbb{S}^2\to q(d\mathbb{S}^2)$ is a $C^\infty$-diffeomorphism. Then Lemma \ref{lemma4.1} is proved for $\tilde{\psi}_Q\vert_{U_{Q,i}}$ and $b\vert_{U_i}$. To complete the proof it is sufficient to show that the equality
	\begin{equation*}
		\varphi_i^{-1}\circ\varphi_{Q,i}(f(x_1,x_2,x_3))=\varphi_j^{-1}\circ\varphi_{Q,j}(f(x_1,x_2,x_3)),\quad i,j\in\{1,2\}
	\end{equation*}
	holds for $f(x_1,x_2,x_3)\in U_{Q,i}\cap U_{Q,j}$. This equality is satisfied because the patching relations for $\{(U_{Q,i},\varphi_{Q,i})\}$ are the same as those for the standard atlas of $B$.
\end{proof}
\begin{rem}
	From Lemma \ref{lemma4.1}, the map $\tilde{\psi}_Q$ is a map of blow up type.
\end{rem}
Now using Lemma \ref{lemma4.1} we obtain concrete normal forms for generic singularities and exact locations of the pedal point $Q$ for such singularities. First we recall the following definitions and lemmas.
\begin{defn}
	Let $f,g:(I,s_0)\to\mathbb{R}^n$ be two curve germs. Then $f$ and $g$ are said to be $C^r$ $\mathcal{L}$-equivalent if there exists a $C^r$ diffeomorphism germ $\psi:(\mathbb{R}^n,f(s_0))\to(\mathbb{R}^n,g(s_0))$ such that $g=\psi\circ f$. Moreover, these curve germs are said to be $C^r$ $\mathcal{A}$-equivalent provided that there exist two $C^r$ diffeomorphism germs $\phi:(I,s_0)\to(I,s_0)$ and $\psi:(\mathbb{R}^n,f(s_0))\to(\mathbb{R}^n,g(s_0))$ such that $g\circ \varphi=\psi\circ f$. 
\end{defn}
\begin{defn}[{\cite{brugib}}]
	A function $f:I\to\mathbb{R}$ is said to have $A_k$-type singularity ($k\geq0$) at $s_0\in I$ if $f(s_0)=f'(s_0)=\cdots=f^{(k)}(s_0)=0$ and $f^{(k+1)}(s_0)\neq 0$. 
\end{defn}
\begin{lem}[{\cite[Theorem 3.3]{brugib}}]\label{lem4.3}
	Suppose that $g:(\mathbb{R},0)\to \mathbb{R}$ is a $C^\infty$ function-germ. Suppose further that $g$ has $A_{k}$-type singularity at $0$. In this case, there exists a $C^\infty$ diffeomorphism germ $f:(\mathbb{R},0)\to(\mathbb{R},0)$ such that $g(f(s))=\pm s^k$, where we have $+$ or $-$ depending on the sign of $g^{(k+1)}(0)$.
\end{lem}
\begin{lem}[{\cite[Hadamard's Lemma]{brugib}}]\label{lem4.4}
	Let $f:(\mathbb{R},0)\to \mathbb{R}$ be smooth, and suppose $f^{(p)}(0)=0$ for all $p$ with $1\leq p\leq k$. Then there is a smooth function $f_1:(\mathbb{R},0)\to\mathbb{R}$ such that $f(s)=f(0)+s^{k+1}f_1(s)$ for all $s$ in some neighbourhood of $0$. Moreover, if $f^{(k+1)}(0)\neq 0$ then $f_1(0)\neq 0$.
\end{lem}
\begin{thm}\label{thm4.2}
	Consider a spacelike Legendre curve  $(\mathbf{r}_h,v_h)$ with a spacelike hyperbolic Legendre curvature $(\ell_h,m_h)$. Let $s_0\in I$ such that $m_h(s_0)\neq0$, and let $Q$ be a point in $\mathcal{H}^2_+$. Suppose that $\ell_h$ has an $A_{k-1}$-type singularity at $s_0\in I$. Then the following statements are satisfied:
	\begin{enumerate}
		\item Assume that $Q\in \mathcal{H}^2_+-\{ \mathbf{r}_h(s_0) \}$. The map-germ $\mathcal{P}ed_Q(\mathbf{r}_h):(I,s_0)\to \big(\mathcal{H}^2_+,\mathcal{P}ed_Q(\mathbf{r_h})(s_0)\big)$ is smooth, which means that it is $C^\infty$ $\mathcal{A}$-equivalent to the map-germ $(\mathbb{R},0)\to(\mathbb{R}^2,0)$ defined by $t\mapsto (t,0)$.
		%\item If $Q\in \{ \pm \mathbf{r}_h(s_0) \}$ and $\mathbf{r}_h(s_0)$ is a regular point of $\mathbf{r}_h$, then the map-germ $\mathcal{P}ed_Q(\mathbf{r_h}):(I,s_0)\to \big(\mathcal{H}^2,\mathcal{P}ed_Q(\mathbf{r_h})(s_0)\big)$ is $C^\infty$ left equivalent to the map-germ $(\mathbb{R},0)\to(\mathbb{R}^2,0)$ given by $t\mapsto(t^2,t^3)$.
		\item Assume that $Q=\mathbf{r}_h(s_0)$. Then the map-germ $\mathcal{P}ed_Q(\mathbf{r_h}):(I,s_0)\to \big(\mathcal{H}^2_+,\mathcal{P}ed_Q(\mathbf{r_h})(s_0)\big)$ is $C^1$ $\mathcal{A}$-equivalent to the map-germ $(\mathbb{R},0)\to(\mathbb{R}^2,0)$ defined by $t\mapsto(t^{k+2},t^{k+3})$.
		%\item If $Q\in \{ \pm \mathbf{r}_h(s_0) \}$, then the map-germ $\mathcal{P}ed_Q(\mathbf{r_h}):(I,s_0)\to \big(\mathcal{H}^2,\mathcal{P}ed_Q(\mathbf{r_h})(s_0)\big)$ is $C^\infty$ left equivalent to the map-germ $(\mathbb{R},0)\to(\mathbb{R}^2,0)$ given by $t\mapsto(t^{j+1},t^{j+2})$ if and only if the map-germ $\mathbf{r}_h:(I,s_0)\to\big(\mathcal{H}^2,\mathbf{r}_h(s_0)\big)$ is left equivalent to the map germ $(\mathbb{R},0)\to(\mathbb{R}^2,0)$ given by $t\mapsto(t^j,t^{j+1})$, where $j\geq 2$.
	\end{enumerate} 
\end{thm}
\begin{proof}
	\begin{enumerate}
		\item Since $Q\in \mathcal{H}^2_+-\{ \mathbf{r}_h(s_0) \}$ and $m_h(s_0)\neq0$, we have ${\mathcal{P}}ed'_Q(\mathbf{r}_h)(s_0)\neq 0$ from Theorem \ref{thm3.3}. Therefore the map-germ ${\mathcal{P}}ed_Q(\mathbf{r}_h)(s_0)$ is non-singular.
		\item By an appropriate rotation of $\mathcal{H}^2_+$, we may assume that $Q=(1,0,0)\in\mathcal{H}^2_+$ and $\mathbf{r}_h(s_0)=( 1,0,0)$, $v_h(s_0)=(0,1,0)$ and $\mu_h(s_0)=(0,0, 1)$. Then it is easy to see that the component function-germs $v_1,v_2$, and $v_3$ of the map-germ $v_h=(v_1,v_2,v_3):(I,s_0)\to d\mathcal{S}^2$ admit the lowest degree of non-zero terms as $k+2,0$ and $1$, respectively.
		
		Then using the hyperbolic Legendrian Frenet-Serret type formulas we may take the map-germ $v_h:(I,s_0)\to(d\mathcal{S}^2,v_h(s_0))$ 
		\[ v_h(s)=\begin{pmatrix}
			\dfrac{1}{(k+2)!}\ell_h^{(k)}(s_0)m_h(s_0)(s-s_0)^{k+2}+C(s-s_0) \\
			1+A(s-s_0)\\
			m_h(s_0)(s-s_0)+B(s-s_0)
		\end{pmatrix}, \]
		where $A,B$, and $C$ are some $C^\infty$ function-germs $(\mathbb{R},0)\to(\mathbb{R},0)$ such that $\dfrac{d^p A}{dt^p}(0)=\dfrac{d^p B}{dt^p}(0)=0$ ($p=0,1$) and $\dfrac{d^p C}{dt^p}(0)=0$ ($p\leq k+2$). From Lemma \ref{lemma4.1}, we have 
		\[ \varphi_{Q,1}(f(v_1,v_2,v_3))=(\tanh(\lambda)v_2,\dfrac{v_3}{v_2}), \]
		where $\sinh(\lambda)=v_1$. This directly gives
		\begin{equation}\label{eq1}
			\varphi_{Q,1}(f(v_h(s)))=\begin{pmatrix}
				\tanh(\lambda)\,(1+A(s-s_0))\\
				\dfrac{m_h(s_0)(s-s_0)+B(s-s_0)}{(1+A(s-s_0))}
			\end{pmatrix}.
		\end{equation} 
		From \eqref{eq4.1}, we have $q\circ \tilde{\psi}_Q\circ \varphi_{Q,1}^{-1}(u_1,u_2)=(u_1,u_2u_1)$ which yields that the map-germ $\tilde{\psi}_Q\circ v_h:(I,s_0)\to(d\mathcal{S}^2,\tilde{\psi}_Q\circ v_h(s_0))$ is $C^\infty$ $\mathcal{A}$-equivalent to
		\[ \begin{pmatrix}
			\tanh(\lambda)\,(1+A(s-s_0))\\
			\tanh(\lambda)\,(m_h(s_0)(s-s_0)+B(s-s_0))
		\end{pmatrix}. \]
		Thus the-map germ $\mathcal{P}ed_Q(\mathbf{r_h}):(I,s_0)\to (\mathcal{H}^2_+,\mathcal{P}ed_Q(\mathbf{r_h})(s_0))$ is $C^\infty$ $\mathcal{A}$-equivalent to 
		\begin{equation*} 
			\begin{pmatrix}
				\dfrac{1}{(k+2)!}\ell_h^{(k)}(s_0)m_h(s_0)(s-s_0)^{k+2}+\hat{C}(s-s_0)\smallskip\\
				\dfrac{1}{(k+2)!}\ell_h^{(k)}(s_0)(m_h(s_0))^2(s-s_0)^{k+3}+\hat{B}(s-s_0)
			\end{pmatrix},
		\end{equation*} 
		where $\hat{B}$ and $\hat{C}$ are certain $C^\infty$ function-germs $(\mathbb{R},0)\to(\mathbb{R},0)$ such that $\hat{B}$ has at least $A_{k+3}$ singularity at $0$ and $\hat{C}$ has at least $A_{k+2}$ singularity at $0$. 
		From Lemma \ref{lem4.3} we conclude that $\mathcal{P}ed_Q(\mathbf{r}_h)$ is $C^\infty$ $\mathcal{A}$-equivalent to 
		\[ (t^{k+2}+D_1(t),\:t^{k+3}+D_2(t)), \]
		where $D_i:(\mathbb{R},0)\to(\mathbb{R},0)$ are some $C^\infty$ function germs with $\dfrac{d^pD_i}{dt^p_i}(0)=0$ for $p\leq k+i+1$. 
		
		We consider the following cases.
		
		Let $k+2$ be odd. In this case we need to consider
		\[ h(x_1,x_2)=\big(x_1+D_1({x_1}^{\frac{1}{k+2}}),x_2+D_2({x_1}^{\frac{1}{k+2}})\big). \]
		On the other hand, if $k+3$ is odd, then it is reasonable to consider
		\[ h(x_1,x_2)=\big(x_1+D_1({x_2}^{\frac{1}{k+3}}),x_2+D_2({x_2}^{\frac{1}{k+3}})\big). \]
		Now all we need to do is to show that $h$ is a germ of $C^1$ diffeomorphism for both cases. One can see that both of the maps $x_p\mapsto D_i({x_p}^{\frac{1}{k+p+1}})$ and $x_p\mapsto \dfrac{d D_i({x_p}^{\frac{1}{k+p+1}})}{d x_p}(x_p)$ are well-defined and continuous even at $0$. Moreover it is not hard to see that the Jacobian matrix of $h$ at $(0,0)$ is the unit matrix. Thus $h$ is a germ of $C^1$ diffeomorphism which concludes the proof.
	\end{enumerate}
\end{proof}
\begin{rem}
	As an immediate consequence of Theorem \ref{thm4.2}, if $Q= \mathbf{r}_h(s_0)$ and $\mathbf{r}_h(s_0)$ is a regular point of $\mathbf{r}_h$, then the map-germ $\mathcal{P}ed_Q(\mathbf{r_h}):(I,s_0)\to \big(\mathcal{H}^2_+,\mathcal{P}ed_Q(\mathbf{r_h})(s_0)\big)$ is $C^1$ $\mathcal{A}$-equivalent to the map-germ $(\mathbb{R},0)\to(\mathbb{R}^2,0)$; $t\mapsto(t^2,t^3)$.
\end{rem}
We present an example to Theorem \ref{thm4.2}.
\begin{ex}
	Let $\mathbf{r}_h:I\to\mathcal{H}^2_+$ be the hyperbolic $3/2$ cusp given by $\mathbf{r}_h(s)=(\sqrt{1+s^4+s^6},\,s^2,\, s^3)$. By differentiating $\mathbf{r}_h$ with respect to $s$, we find that
	\[ {\mathbf{r}}'_h(s)=(\dfrac{2s^3+3s^5}{\sqrt{1+s^4+s^6}},\,2s,\,3s^2). \]
	If we take $v_h:I\to d\mathcal{S}^2$ 
	\[v_h(s)=\dfrac{1}{\sqrt{s^6+9s^2+4}}\big(s^3\sqrt{1+s^4+s^6},s^5+3s,s^6-2\big), \]
	then we obtain a spacelike Legendre immersion $(\mathbf{r}_h,v_h):I\to \Delta_1$ with spacelike hyperbolic curvature $(\ell_h,m_h)$, where
	\[\ell_h(s)=\dfrac{s\sqrt{s^6+9s^2+4}}{\sqrt{1+s^4+s^6}}, \qquad m_h(s)=\dfrac{s^{10}+15s^6+10s^4+6}{(s^6+9s^2+4)\sqrt{1+s^4+s^6}}\neq 0. \]
	Furthermore, using the Lorentz vector product we obtain $\mu_h:I\to d\mathcal{S}^2$ 
	\[\mu_h(s)=\dfrac{\sqrt{1+s^4+s^6}}{\sqrt{s^6+9s^2+4}}\big(\dfrac{2s^2+3s^4}{\sqrt{1+s^4+s^6}},2,3s \big). \]
	Choose the point $Q_1=(\sqrt{2},1,0)\in \mathcal{H}^2_+ -\{\mathbf{r}_h(0)\}$. Then by Theorem \ref{thm4.2} (1) the map-germ $\mathcal{P}ed_{Q_1}(\mathbf{r_h}):(I,0)\to \big(\mathcal{H}^2_+,(\sqrt{2},1,0)\big)$ is smooth. The hyperbolic pedal curve of $\mathbf{r}_h$ with respect to $Q_1$ is illustrated in Fig. \ref{fig2}.
	\begin{figure}[H]
		\centering
		\includegraphics[width=0.5\textwidth]{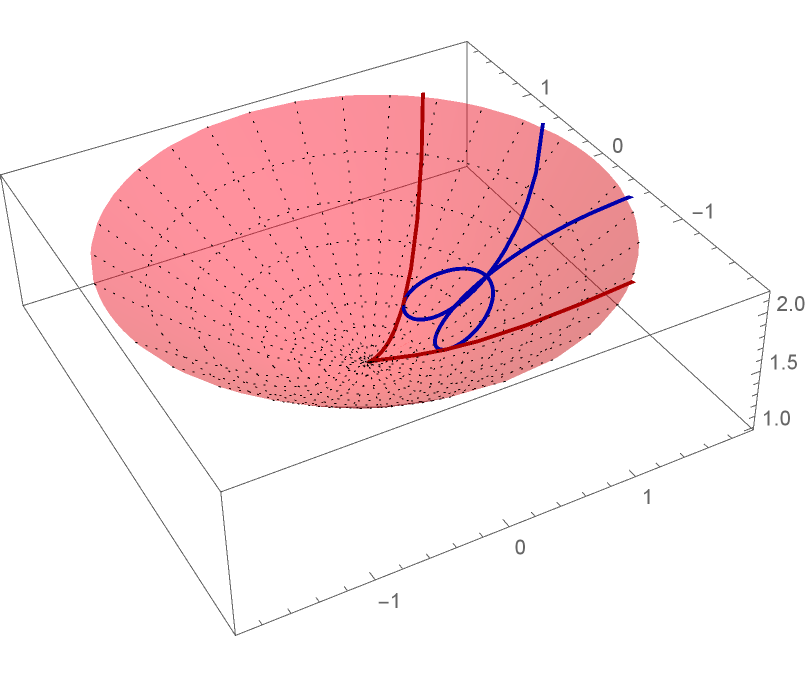}
		\caption{The hyperbolic $3/2$ cusp $\mathbf{r}_h$ (red) and its hyperbolic pedal curve relative to $Q_1=(\sqrt{2},1,0)$ (blue).} \label{fig2}
	\end{figure}
	Now choose $Q_2=(\sqrt{3},1,1)= \mathbf{r}_h(1)$, where $\mathbf{r}_h(1)$ is a regular point of $\mathbf{r}_h$. By Theorem \ref{thm4.2} (2) we have that the map-germ $\mathcal{P}ed_{Q_2}(\mathbf{r_h}):(I,1)\to \big(\mathcal{H}^2_+,(\sqrt{3},1,1)\big)$ is $C^1$ $\mathcal{A}$-equivalent to the map-germ defined by $t\mapsto(t^2,t^3)$. This means that the pedal curve $\mathcal{P}ed_{Q_2}(\mathbf{r}_h)$ has a $3/2$ cusp at $s_0=1$. The hyperbolic pedal of $\mathbf{r}_h$ with respect to $Q_2$ is given in Fig. \ref{fig4}.
	\begin{figure}[H]
		\centering
		\includegraphics[width=0.5\textwidth]{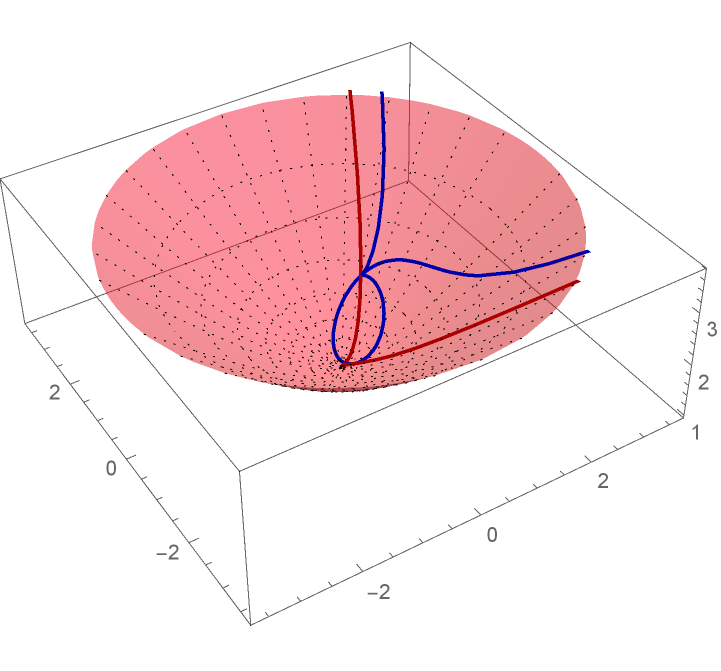}
		\caption{The hyperbolic $3/2$ cusp $\mathbf{r}_h$ (red) and its hyperbolic pedal curve relative to $Q_2=(\sqrt{3},1,1)$ (blue).} \label{fig4}
	\end{figure} 
	Let us take $Q_3=(1,0,0)=\mathbf{r}_h(0)$. It is easy to see that by using the canonical projection $q$, the map-germ $\mathbf{r}_h:(I,0)\to(\mathcal{H}^2_+,(1,0,0))$ is $C^1$ $\mathcal{A}$-equivalent to the map-germ defined by $s\mapsto (s^2,s^3)$. We find that $\ell_h(0)=0$ and $\ell_h'(0)=2\neq0$ which means that $\ell_h$ has an $A_0$-type singularity at $s_0=0$. Thus by Theorem \ref{thm4.2} (2) the map-germ $\mathcal{P}ed_{Q_3}(\mathbf{r_h}):(I,0)\to \big(\mathcal{H}^2_+,(1,0,0)\big)$ is $C^1$ $\mathcal{A}$-equivalent to the map-germ $s\mapsto(s^3,s^4)$. The hyperbolic pedal of $\mathbf{r}_h$ with respect to $Q_3$ is illustrated in Fig. \ref{fig5}.
	\begin{figure}[H]
		\centering
		\includegraphics[width=0.5\textwidth]{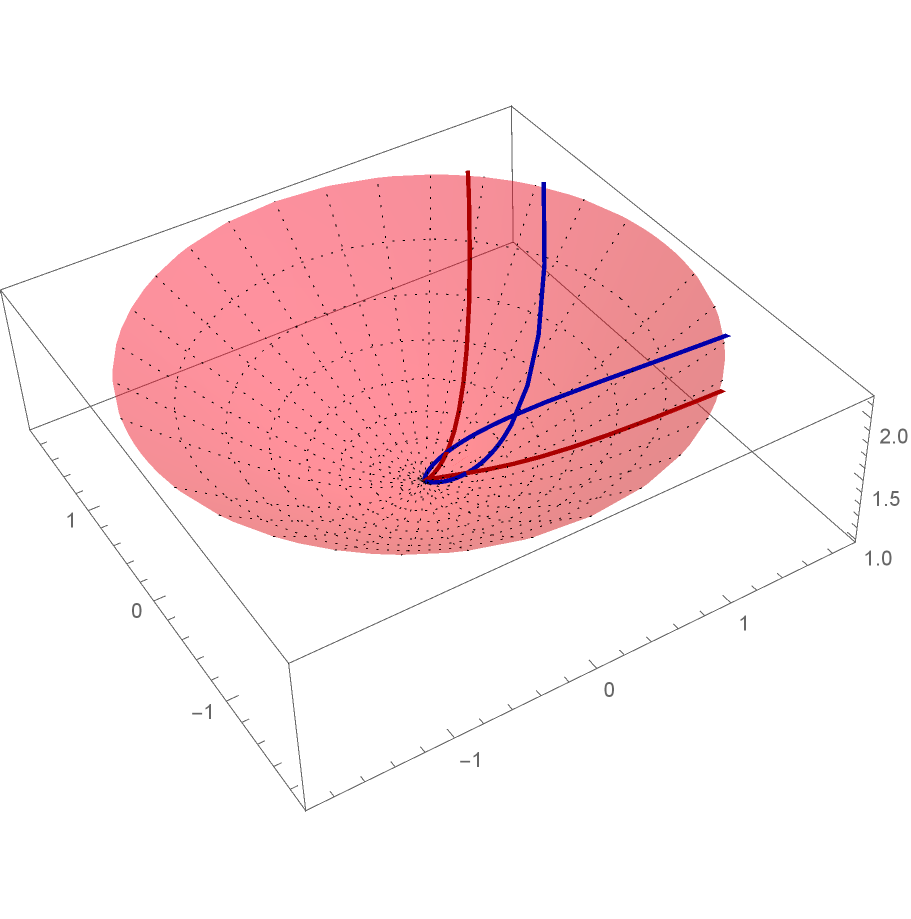}
		\caption{The hyperbolic $3/2$ cusp $\mathbf{r}_h$ (red) and its hyperbolic pedal curve relative to $Q_3=(1,0,0)$ (blue).} \label{fig5}
	\end{figure}
\end{ex}
Now we investigate singularities of pedal curves for points $s_0\in I$ such that $m_h(s_0)=0$. 

Let $s_0\in I$. For any $s$ such that $s+s_0\in I$, define
\[ \gamma_j(s)=(m_h(s+s_0),m_h'(s+s_0),\ldots,m_h^{(j-1)}(s+s_0)),\quad j\geq 1 \]
Then $\gamma^*_j\mathcal{M}_j\mathcal{E}_1$ is an ideal of $\mathcal{E}_1$. We take in consideration quotient $\mathcal{E}_1$ algebras given by $\mathcal{E}_1/(\gamma^*_j\mathcal{M}_j\mathcal{E}_1)$.
\begin{lem}\label{lemma4.4}
	Let $\ell_h$ has an $A_{k-1}$-singularity ($k\geq0$) at $s+s_0$, $s_0\in I$, where $A_{-1}$ means that $\ell_h(s+s_0)\neq0$. Then the followings are satisfied:
	\begin{enumerate}
		\item $\langle v_h^{(j+1)}(s+s_0),v_h(s+s_0)\rangle\in \gamma^*_j\mathcal{M}_j\mathcal{E}_1$.
		\item $\langle v_h^{(j+1)}(s+s_0),\mu_h(s+s_0)\rangle+\gamma^*_j\mathcal{M}_j\mathcal{E}_1=m_h^{(j)}(s+s_0)+\gamma^*_j\mathcal{M}_j\mathcal{E}_1$.
		\item $\langle v_h^{(j+k+2)}(s+s_0),\mathbf{r}_h(s+s_0)\rangle+\gamma^*_j\mathcal{M}_j\mathcal{E}_1=-\binom{j+k+1}{j}\,\ell_h^{(k)}(s+s_0)\, m_h^{(j)}(s+s_0)+\gamma^*_j\mathcal{M}_j\mathcal{E}_1$.
	\end{enumerate}
\end{lem} 
\begin{proof}
	For simplicity, we use just the notation $f$ instead of $f(s+s_0)$. We proceed by induction on $j$. 
	
	For $j=1$, it is enough to show that the followings are satisfied.
	\begin{align*}
		&\langle v_h'',v_h\rangle=-m_h^2,\\
		&\langle v_h'',\mu_h\rangle=m_h',\\
		&\langle v_h^{(k+3)},\mathbf{r}_h\rangle=-(k+2)m_h'\ell_h^{(k)}-m_h\ell_h^{(k+1)}.
	\end{align*}
	Since $\langle v_h,v_h\rangle=1$, we have $\langle v_h',v_h\rangle=0$. Taking derivative of this equality we see that $\langle v_h'',v_h\rangle+\langle v_h',v_h'\rangle=0$ which yields $\langle v_h'',v_h\rangle=-m_h^2$. By \eqref{eq2.3} we know that $\langle v_h',\mu_h\rangle=m_h$. Then we find that $\langle v_h'',\mu_h\rangle+\langle v_h',\mu_h'\rangle=m_h'$ which yields $\langle v_h'',\mu_h\rangle=m_h'$. After direct calculations we get
	\begin{align*}
		v_h^{(k+3)}=&\:m_h^{(k+2)}\mu_h+\binom{k+2}{1}m_h^{(k+1)}(\ell_h\mathbf{r}_h-m_hv_h)+\binom{k+2}{2}m_h^{(k)}(\ell_h\mathbf{r}_h-m_hv_h)'\\
		&+\cdots+m_h(\ell_h\mathbf{r}_h-m_hv_h)^{(k+1)}.
	\end{align*}
	Hence using $A_{k-1}$ singularity of $\ell_h$ we find that $\langle v_h^{(k+3)},\mathbf{r}_h\rangle=-(k+2)m_h'\ell_h^{(k)}-m_h\ell_h^{(k+1)}$.
	
	Now we prove this lemma for $j=i+1$ by assuming that the lemma is satisfied for $j\leq i$. Taking derivative of $\langle v_h^{(i+1)},v_h\rangle\in \gamma^*_i\mathcal{M}_i\mathcal{E}_1$ yields 
	\begin{align*}
		\langle v_h^{(i+2)},v_h\rangle+\langle v_h^{(i+1)},v_h'\rangle\in\gamma^*_{i+1}\mathcal{M}_{i+1}\mathcal{E}_1.
	\end{align*}
	Since $\langle v_h^{(i+1)},v_h'\rangle=m_h\langle v_h^{(i+1)},\mu_h\rangle\in\gamma^*_{i}\mathcal{M}_{i}\mathcal{E}_1$ and $\gamma^*_{i}\mathcal{M}_{i}\mathcal{E}_1\subset\gamma^*_{i+1}\mathcal{M}_{i+1}\mathcal{E}_1$, we obtain $\langle v_h^{(i+2)},v_h\rangle\in\gamma^*_{i+1}\mathcal{M}_{i+1}\mathcal{E}_1$. 
	
	By differentiating $\langle  v_h^{(i+1)},\mu_h\rangle+\gamma^*_i\mathcal{M}_i\mathcal{E}_1=m_h^{(i)}+\gamma^*_i\mathcal{M}_i\mathcal{E}_1$ and using \eqref{eq2.3}, we deduce
	\begin{equation}\label{lemeq1}
		\langle v_h^{(i+2)},\mu_h\rangle+\ell_h\langle v_h^{(i+1)},\mathbf{r}_h\rangle-m_h\langle v_h^{(i+1)},v_h\rangle+\gamma^*_{i+1}\mathcal{M}_{i+1}\mathcal{E}_1=m_h^{(i+1)}+\gamma^*_{i+1}\mathcal{M}_{i+1}\mathcal{E}_1
	\end{equation}
	From the statement (1) of this lemma, we have $\langle v_h^{(i+1)},v_h\rangle\in\gamma^*_i\mathcal{M}_i\mathcal{E}_1$. Now we consider two cases.
	
	Let $\ell_h\neq 0$ i.e. $k=0$. Taking $j=i-1$ ($i\geq1$) in the statement (3) of this lemma, we find that the statement (2) is satisfied for $j=i+1$. 
	
	Now let $\ell_h=0$ i.e. $k\neq 0$. By substituting $\ell_h=0$ into \eqref{lemeq1} we find that $\langle v_h^{(i+2)},\mu_h\rangle+\gamma^*_{i+1}\mathcal{M}_{i+1}\mathcal{E}_1=m_h^{(i+1)}+\gamma^*_{i+1}\mathcal{M}_{i+1}\mathcal{E}_1$.
	Hence we see that the statement (2) of this lemma is satisfied for $j=i+1$.
	
	Finally, by differentiating $\langle v_h^{(i+k+2)},\mathbf{r}_h\rangle+\gamma^*_i\mathcal{M}_i\mathcal{E}_1=-\binom{i+k+1}{i}\,\ell_h^{(k)}\, m_h^{(i)}+\gamma^*_i\mathcal{M}_i\mathcal{E}_1$, we find that
	\begin{align*}
		&\langle v_h^{(i+k+3)},\mathbf{r}_h\rangle+\ell_h\langle v_h^{(i+k+2)},\mu_h\rangle+\gamma^*_{i+1}\mathcal{M}_{i+1}\mathcal{E}_1 \\
		&\qquad\qquad =-\binom{i+k+1}{i}\,(\ell_h^{(k+1)}\,m_h^{(i)}+\ell_h^{(k)}m_h^{(i+1)})+\gamma^*_{i+1}\mathcal{M}_{i+1}\mathcal{E}_1.
	\end{align*}
	Therefore considering two cases based on $\ell_h$ again one can conclude the proof.
\end{proof}
%\begin{rem} In particular, if the curve is regular at $s+s_0$ we know that $\ell(s+s_0)\neq 0$ which yields
%	$$\langle v_h^{(j+2)}(s+s_0),\mathbf{r}_h(s+s_0)\rangle+\gamma^*_j\mathcal{M}_j\mathcal{E}_1=-(j+1)\ell_h(s+s_0) m_h^{(j)}(s+s_0)+\gamma^*_j\mathcal{M}_j\mathcal{E}_1.$$
%\end{rem}
\begin{thm}\label{thm4.6}
	Let $(\mathbf{r}_h,v_h)$ be a spacelike Legendre curve with a spacelike hyperbolic Legendre curvature $(\ell_h,m_h)$, and let $Q\in\mathcal{H}^2_+$ be a point. Suppose that $m_h$ has an $A_{j-1}$ singularity and $\ell_h$ has an $A_{k-1}$ singularity at $s_0\in I$. Then the following statements are satisfied.
	\begin{enumerate}
		\item Let $Q= \mathbf{r}_h(s_0)$. The map-germ $\mathcal{P}ed_Q(\mathbf{r}_h):(I,s_0)\to \big(\mathcal{H}^2_+,\mathcal{P}ed_Q(\mathbf{r}_h)(s_0)\big)$ is $C^1$ $\mathcal{A}$-equivalent to the map-germ $(\mathbb{R},0)\to(\mathbb{R}^2,0)$; $t\mapsto(t^{j+k+2},t^{2j+k+3})$.
		\item If $Q\in G_{v_h(s_0)}-\{\mathbf{r}_h(s_0)\}$, then the map-germ $\mathcal{P}ed_Q(\mathbf{r}_h):(I,s_0)\to (\mathcal{H}^2_+,\mathcal{P}ed_Q(\mathbf{r}_h)(s_0))$ is $C^1$ $\mathcal{A}$-equivalent to the map-germ $(\mathbb{R},0)\to(\mathbb{R}^2,0)$; $t\mapsto(t^{j+1},t^{2j+k+3})$.
		\item If $Q\in \mathcal{H}^2_+- G_{v_h(s_0)}$, then the map-germ $\mathcal{P}ed_Q(\mathbf{r}_h):(I,s_0)\to (\mathcal{H}^2_+,\mathcal{P}ed_Q(\mathbf{r}_h)(s_0))$ is $C^1$ $\mathcal{A}$-equivalent to the map-germ $(\mathbb{R},0)\to(\mathbb{R}^2,0)$; $t\mapsto(t^{j+1},t^{j+k+2})$.
	\end{enumerate}
\end{thm}
\begin{proof}
	\begin{enumerate}
		\item By a suitable rotation of $\mathcal{H}^2_+$, we may assume that $Q=(1,0,0)\in\mathcal{H}^2_+$, $\mathbf{r}_h(s_0)=(1,0,0)$, $v_h(s_0)=(0,1,0)$, and $\mu_h=(0,0,1)$ since $Q= \mathbf{r}_h(s_0)$. By Lemma \ref{lemma4.4}, we can take the map-germ $v_h:(I,s_0)\to(d\mathcal{S}^2,v_h(s_0))$ 
		\[ v_h(s)=\begin{pmatrix}
			\dfrac{\binom{j+k+1}{j}}{(j+k+2)!}\ell_h^{(k)}(s_0)m_h^{(j)}(s_0)(s-s_0)^{j+k+2}+C(s-s_0) \\
			1+A(s-s_0)\\
			\dfrac{1}{(j+1)!}m_h^{(j)}(s_0)(s-s_0)^{j+1}+B(s-s_0)
		\end{pmatrix}, \]
		where $A,B$, and $C$ are some $C^\infty$ function germs $(\mathbb{R},0)\to(\mathbb{R},0)$. Furthermore, $A$ and $B$ have $A_{j+1}$ singularity at $0$ while $C$ has at least $A_{j+k+1}$ singularity at $0$. By Lemma \ref{lemma4.1} we know that
		\[\varphi_{Q,1}(f(x_1,x_2,x_3))=(\tanh(\lambda)x_2,\dfrac{x_3}{x_2}),\quad \sinh\lambda = x_1,\]
		which yields
		\begin{equation}\label{eq4.2}
			\varphi_{Q,1}(f(v_h(s)))=\begin{pmatrix}
				\tanh(\lambda)\,(1+A(s-s_0))\\
				\dfrac{ m_h^{(j)}(s_0)(s-s_0)^{j+1}+(j+1)!\,B(s-s_0)}{(j+1)!\,(1+A(s-s_0))}
			\end{pmatrix}.
		\end{equation} 
		Since $q\circ \tilde{\psi}_Q\circ \varphi_{Q,1}^{-1}(u_1,u_2)=(u_1,u_2u_1)$, it follows from \eqref{eq4.2} that the map-germ $\tilde{\psi}_Q\circ v_h:(I,s_0)\to(d\mathcal{S}^2,\tilde{\psi}_Q\circ v_h(s_0))$ is $C^\infty$ $\mathcal{A}$-equivalent to
		\[ \begin{pmatrix}
			\tanh(\lambda)\,(1+A(s-s_0))\\
			\tanh(\lambda)\,\dfrac{ m_h^{(j)}(s_0)(s-s_0)^{j+1}+(j+1)!\,B(s-s_0)}{(j+1)!}
		\end{pmatrix}. \]
		Since the first component of $v_h$ is $\sinh\lambda$ we see that the-map germ $\mathcal{P}ed_Q(\mathbf{r_h}):(I,s_0)\to (\mathcal{H}^2_+,\mathcal{P}ed_Q(\mathbf{r_h})(s_0))$ is $C^\infty$ $\mathcal{A}$-equivalent to 
		\begin{equation*}\label{eq4.3}
			\begin{pmatrix}
				\dfrac{\binom{j+k+1}{j}}{(j+k+2)!}\ell_h^{(k)}(s_0)m_h^{(j)}(s_0)(s-s_0)^{j+k+2}+\hat{C}(s-s_0)\\
				\dfrac{\binom{j+k+1}{j}}{(j+k+2)!(j+1)!}\ell_h^{(k)}(s_0)(m_h^{(j)}(s_0))^2(s-s_0)^{2j+k+3}+\hat{B}(s-s_0)
			\end{pmatrix},
		\end{equation*} 
		where $\hat{B}$ and $\hat{C}$ are certain $C^\infty$ function-germs $(\mathbb{R},0)\to(\mathbb{R},0)$. Moreover, $\hat{B}$ has at least $A_{2j+k+3}$ singularity at $0$ while $\hat{C}$ has at least $A_{j+k+2}$ singularity at $0$. 
		
		Since $\dfrac{\binom{j+k+1}{j}}{(j+k+2)!}\ell_h^{(k)}(s_0)m_h^{(j)}(s_0)\neq 0 $ and  $\hat{C}(s-s_0)$ has at least $A_{j+k+2}$-type singularity at $0$, by using Lemma \ref{lem4.3} we find that $\mathcal{P}ed_Q(\mathbf{r}_h)$ is $C^\infty$ $\mathcal{A}$-equivalent to 
		\[ (t^{j+k+2},\:t^{2j+k+3}+D(t)), \]
		%	where $\hat{B}$ and $\hat{C}$ are certain $C^\infty$ function-germs $(\mathbb{R},0)\to(\mathbb{R},0)$. Furthermore, $\hat{B}$ has an $A_{2j+k+3}$ singularity at $0$ while $\hat{C}$ has an $A_{j+1}$ singularity at $0$. By using Lemma \ref{lem4.3} and Lemma \ref{lem4.4} we obtaing that $\mathcal{P}ed_Q(\mathbf{r}_h)$ is $C^\infty$ $\mathcal{A}$-equivalent to
		%\[ (t^{j+1},\:t^{2j+k+3}+D(t)), \]
		where $D:(\mathbb{R},0)\to(\mathbb{R},0)$ is a $C^\infty$ function germ with $\dfrac{d^pD}{dt^p}(0)=0$ for $p\leq 2j+k+3$. 
		
		We have two cases:
		\begin{enumerate}
			\item Let $2j+k+3$ be odd. Consider the map
			\[ h_2(x_1,x_2)=\big(x_1,x_2+D(x_2^{\frac{1}{2j+k+3}})\big). \]
			Since $2j+k+3$ is odd, we see that $x_2\mapsto D(x_2^{\frac{1}{2j+k+3}})$ is well-defined and continuous everywhere. In addition by Lemma \ref{lem4.4} there exists a $C^\infty$ function-germ  $\tilde{D}:(\mathbb{R},0)\to(\mathbb{R},0)$ such that ${D}(t)=t^{2j+k+4}\tilde{D}(t)$. Therefore we get
			\begin{align*}
				\dfrac{dD(x_2^{\frac{1}{2j+k+3}})}{dx_2}=&\lim\limits_{h\to 0}\dfrac{D((x_2+h)^{\frac{1}{2j+k+3}})-D(x_2^{\frac{1}{2j+k+3}})}{h} \\
				=&\lim\limits_{h\to 0}\dfrac{(x_2+h)^{\frac{2j+k+4}{2j+k+3}}\tilde{D}((x_2+h)^{\frac{1}{2j+k+3}})-x_2^{\frac{2j+k+4}{2j+k+3}}\tilde{D}(x_2^{\frac{1}{2j+k+3}})}{h}\\
				=&\lim\limits_{h\to 0}\dfrac{(x_2+h)^{\frac{2j+k+4}{2j+k+3}}-x_2^{\frac{2j+k+4}{2j+k+3}}}{h}\tilde{D}((x_2+h)^{\frac{1}{2j+k+3}})\\&+x_2^{\frac{2j+k+4}{2j+k+3}}\lim\limits_{h\to 0}\dfrac{\tilde{D}((x_2+h)^{\frac{1}{2j+k+3}})-\tilde{D}(x_2^{\frac{1}{2j+k+3}})}{h}\\
				=&\dfrac{2j+k+4}{2j+k+3}x_2^{\frac{1}{2j+k+3}}\tilde{D}(x_2^{\frac{1}{2j+k+3}})\\&+x_2^{\frac{2}{2j+k+3}}\tilde{D}'(x_2^{\frac{1}{2j+k+3}})\dfrac{1}{2j+k+3}.
			\end{align*}
			Hence $x_2\mapsto \dfrac{dD(x_2^{\frac{1}{2j+k+3}})}{dx_2}$ is well-defined and continuous everywhere. Since we have $\dfrac{dD(x_2^{\frac{1}{2j+k+3}})}{dx_2}(0)=0$, the Jacobian matrix of $h_2$ at $(0,0)$ is the unit matrix. So $h_2$ is a germ of $C^1$-diffeomorphism.
			\item Let $2j+k+3$ be even. Consider
			\[ h_3(x_1,x_2)=(x_1,x_2+\hat{D}(x_2)), \]
			where 
			\[ \hat{D}(x_2)=\begin{cases}
				D(x_2^{\frac{1}{2j+k+3}})\quad&;\quad x_2\geq0\smallskip\\
				-D((-x_2)^{\frac{1}{2j+k+3}})\quad&;\quad x_2<0.
			\end{cases} \]
			Then $x_2\mapsto\hat{D}(x_2)$ is well-defined and continuous even at $x_2=0$. Furthermore, we find that
			\[ \hat{D}'(x_2)=\begin{cases}
				D'(x_2^{\frac{1}{2j+k+3}})\dfrac{1}{(2j+k+3)x_2^{\frac{2j+k+2}{2j+k+3}}}\quad &;\quad x_2>0\smallskip\\
				D'((-x_2)^{\frac{1}{2j+k+3}})\dfrac{1}{(2j+k+3)(-x_2)^{\frac{2j+k+2}{2j+k+3}}}\quad&;\quad x_2<0\smallskip \\
				0\quad&;\quad x_2=0
			\end{cases} \]
			Then $x_2\mapsto \dfrac{d\hat{D}(x_2^{\frac{1}{2j+k+3}})}{dx_2}$ is also well-defined and continuous everywhere. In addition the Jacobian matrix of $h_3$ at $(0,0)$ is the unit matrix. So $h_3$ is a germ of $C^1$-diffeomorphism. Therefore the map-germ $\mathcal{P}ed_Q(\mathbf{r}_h)$ is $C^1$ $\mathcal{A}$-equivalent to the map-germ given by $t\mapsto(t^{j+k+2},t^{2j+k+3})$.
		\end{enumerate}
		\item Let $Q=(1,0,0)$. By an appropriate rotation of $\mathcal{H}^2_+$, we may assume that $v_h(s_0)=(0,1,0)$. Moreover we can take $\mathbf{r}_h(s_0)=(a,0,b)$ and $\mu(s_0)=(b,0,a)$ such that $a^2-b^2=1$ and $a,0\neq b\in\mathbb{R}$ since $Q\in G_{v_h(s_0)}-\{\mathbf{r}_h(s_0)\}$. Thus by Lemma \ref{lemma4.4} we can take the map-germ $v_h:(I,s_0)\to(d\mathcal{S}^2,v_h(s_0))$ 
		\[ v_h(s)=\begin{pmatrix}
			a\gamma+b\delta\\
			1+A(s-s_0)\\
			b\gamma+a\delta
		\end{pmatrix}, \]
		where 
		\begin{align*}
			\delta(s)&=\dfrac{1}{(j+1)!}m_h^{(j)}(s_0)(s-s_0)^{j+1}+B(s-s_0) \\
			\gamma(s)&=\dfrac{\binom{j+k+1}{j}}{(j+k+2)!}\ell_h^{(k)}(s_0)m_h^{(j)}(s_0)(s-s_0)^{j+k+2}+C(s-s_0)
		\end{align*}
		and $A,B$, and $C$ are some $C^\infty$ function germs $(\mathbb{R},0)\to(\mathbb{R},0)$. Furthermore, $A$ and $B$ have at least $A_{j+1}$ singularity at $0$, while $C$ has at least $A_{j+k+1}$ singularity at $0$. Hence we find that 
		\begin{equation*} 
			\varphi_{Q,1}(f(v_h(s)))=\begin{pmatrix}
				\tanh(\lambda)\,(1+A(s-s_0))\\
				\dfrac{b\gamma+a\delta}{(1+A(s-s_0))}
			\end{pmatrix},
		\end{equation*}
		where $\sinh\lambda=a\gamma+b\delta$. For a linear isomorphism $h_1:\mathbb{R}^2_1\to\mathbb{R}^2_1$ given by $h_1(u_1,u_2)=(u_1,u_2+\frac{b}{a}u_1)$ and a $C^\infty$ diffeomorphism $h_2:\mathbb{R}^2\to\mathbb{R}^2$ given by $h_2(U_1,U_2)=(U_1,U_2+\frac{b}{a}U_1^2)$, it is easy to show that the following is satisfied \cite[Lemma 5.1]{nishimura3}.
		\[ q\circ \tilde{\psi}_Q\circ\varphi_{Q,1}^{-1}\circ h_1(u_1,u_2)=h_2\circ q\circ \tilde{\psi}_Q\circ\varphi_{Q,1}^{-1}(u_1,u_2). \]
		By taking $u_1=\tanh(\lambda)\,(1+A(s-s_0))$ and $u_2=\dfrac{b\gamma+a\delta}{(1+A(s-s_0))}$ we find that $q\circ \tilde{\psi}_Q\circ\varphi_{Q,1}^{-1}\circ h(u_1,u_2)$ is $C^\infty$ $\mathcal{A}$-equivalent to $\mathcal{P}ed_Q(\mathbf{r}_h)$ near $s_0$. Then using Taylor expansions $q\circ \tilde{\psi}_Q\circ\varphi_{Q,1}^{-1}\circ h_1(u_1,u_2)$ can be written as 
		\begin{equation*} 
			\begin{pmatrix}
				b\dfrac{1}{(j+1)!}m_h^{(j)}(s_0)(s-s_0)^{j+1}+\hat{C}(s-s_0)\smallskip\\
				b^2 \dfrac{\binom{j+k+1}{j}}{(j+k+2)!(j+1)!}\ell_h^{(k)}(s_0)(m_h^{(j)}(s_0))^2(s-s_0)^{2j+k+3}+\hat{B}(s-s_0)
			\end{pmatrix},
		\end{equation*}
		where $\hat{B}$ and $\hat{C}$ are some $C^\infty$ function germs $(\mathbb{R},0)\to(\mathbb{R},0)$. Moreover $\hat{B}$ has at least  $A_{2j+k+3}$ singularity at $0$ while $\hat{C}$ has at least $A_{j+1}$ singularity at $0$. By using Lemma \ref{lem4.3} and Lemma \ref{lem4.4}, we find that $\mathcal{P}ed_Q(\mathbf{r}_h)$ is $C^\infty$ $\mathcal{A}$-equivalent to
		\[ (t^{j+1},\:t^{2j+k+3}+D(t)), \]
		where $D:(\mathbb{R},0)\to(\mathbb{R},0)$ is a $C^\infty$ function-germ with $\dfrac{d^pD}{dt^p}(0)=0$ for $p\leq 2j+k+3$. 
		
		Hence the proof follows from similar arguments to the proof of (1).
		\item 
		Let $Q=(1,0,0)$. Since $Q\in \mathcal{H}^2_+- G_{v_h(s_0)}$, we have $\langle Q,v_h(s_0)\rangle\neq 0$. Then we choose $v_h(s_0)=(1,0,\sqrt{2})$, $\mathbf{r}_h(s_0)=(\sqrt{2},0,1)$, and $\mu_h(s_0)=(0,1,0)$. By Lemma \ref{lem4.4} we can take the map-germ $v_h:(I,s_0)\to(d\mathcal{S}^2,v_h(s_0))$ is $C^\infty$
		\begin{align*}
			v_h(s)=\begin{pmatrix}
				1+A(s-s_0)\\
				\dfrac{1}{(j+1)!}m_h^{(j)}(s_0)(s-s_0)^{j+1}+B(s-s_0)\\
				\sqrt{2}+\dfrac{\binom{j+k+1}{j}}{(j+k+2)!}\ell_h^{(k)}(s_0)m_h^{(j)}(s_0)(s-s_0)^{j+k+2}+C(s-s_0)
			\end{pmatrix},
		\end{align*}
		where $A,B$, and $C$ are some $C^\infty$ function germs $(\mathbb{R},0)\to(\mathbb{R},0)$. Moreover $A$ and $B$ have at least $A_{j+1}$ singularity at $0$ while $C$ has at least $A_{j+k+1}$ singularity at $0$. From Lemma \ref{lem4.3}, applying suitable scales and reflections
		along coordinate axes of $\mathbb{R}^2_1$ if necessary, we find that $\mathcal{P}ed_Q(\mathbf{r_h})$ is $C^\infty$ $\mathcal{A}$-equivalent to
		$(t^{j+1},t^{j+k+2}+D(t)), $
		where $D:(\mathbb{R},0)\to(\mathbb{R},0)$ is $C^\infty$ function germ with $\dfrac{d^pD}{dt^p}(0)=0$ for $p\leq j+k+2$. By considering two cases similar to (1), one can conclude the proof.
	\end{enumerate}
\end{proof}
\begin{ex}
	Consider the curve $\mathbf{r}_h(s)=(\sqrt{1+s^6+s^{14}}, s^3,s^7)$ in $\mathcal{H}^2_+$. Differentiating this equation with respect to $s$ yields
	\[ {\mathbf{r}}'_h(s)=(\dfrac{3s^5+7s^{13}}{\sqrt{1+s^6+s^{14}}},\,3s^2,\,7s^6).\]
	Taking $v_h:I\to d\mathcal{S}^2$
	\[v_h(s)=\dfrac{1}{\sqrt{16s^{14}+49s^{8}+9}}\big(4s^7\sqrt{1+s^6+s^{14}},7s^4+4s^{10},4s^{14}-3\big), \]
	we find that $(\mathbf{r}_h,v_h):I\to\Delta_1$ is a spacelike Legendre curve with
	\[\ell_h(s)=\dfrac{s^2\sqrt{16s^{14}+49s^{8}+9}}{\sqrt{1+s^6+s^{14}}}, \qquad m_h(s)=\dfrac{4s^3(16s^{20}+70s^{14}+30s^6+21)}{(16s^{14}+49s^{8}+9)\sqrt{1+s^6+s^{14}}}.\]
	It is easy to show that $m_h(s)$ has an $A_2$ singularity at $0\in I$. Take  $Q=(1,0,0)=\mathbf{r}_h(0)\in\mathcal{H}^2_+$. Then $\ell_h(0)=0$, $\ell'_h(0)=0$, and $\ell_h''(0)\neq 0$. So $\ell_h$ has an $A_1$ singularity at $s_0=0$.  Thus by Theorem \ref{thm4.6} (2) the map-germ $\mathcal{P}ed_Q(\mathbf{r}_h):(I,0)\to \big(\mathcal{H}^2_+,(1,0,0)\big)$ is $C^1$ $\mathcal{A}$-equivalent to the map-germ given by $s\mapsto(s^7,s^{11})$. The hyperbolic pedal curve of $\mathbf{r}_h$ with respect to $Q$ is illustrated in Fig. \ref{fig6}.
	\begin{figure}[H]
		\centering
		\includegraphics[width=0.6\textwidth]{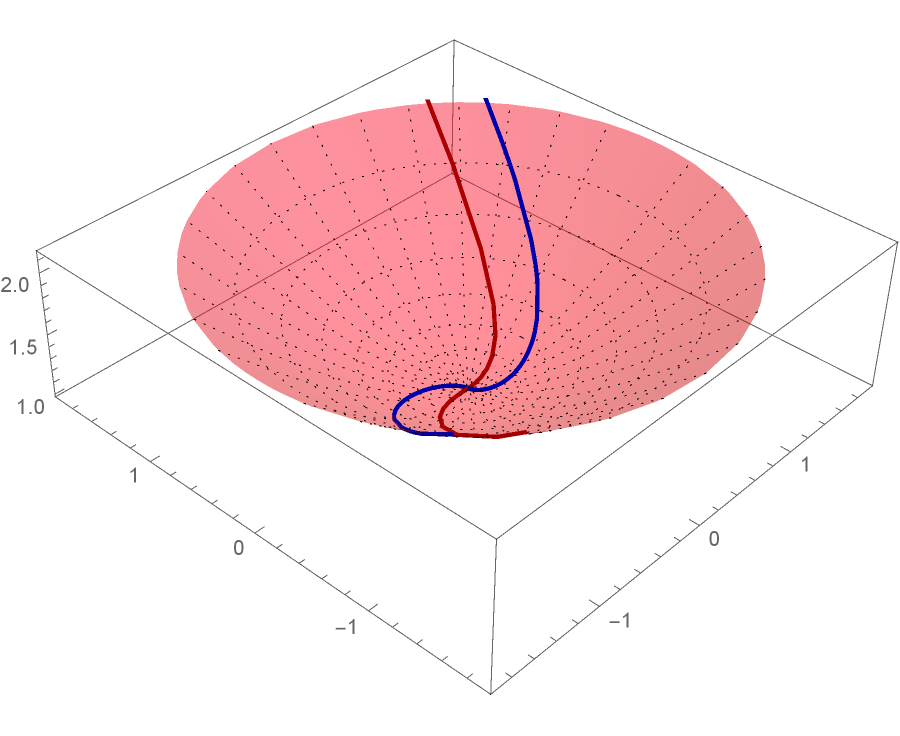}
		\caption{The hyperbolic pedal curve (blue) of $\mathbf{r}_h$ (red) with respect to $Q=(1,0,0)$} \label{fig6}
	\end{figure}
\end{ex}

\section{Conclusions}
We have introduced hyperbolic pedal curves of spacelike frontals in the hyperbolic 2-space. We have then investigated the singularities of these hyperbolic pedal curves for non-singular and singular dual curve germs. 

There are several fruitful research directions we could pursue using the ideas in this paper. 

In a companion paper \cite{tuncer2} we will investigate hyperbolic orthotomics and hyperbolic caustics by using hyperbolic pedal curves defined in this paper. Next we will investigate similar problems for frontals in the de Sitter 2-space. 

In this paper we have focused on the curves. In the future we will consider pedal surfaces of framed surfaces.

\section*{Statements and Declerations}
\subsection*{Conflict of Interests} The authors declare that they have no known conflict of interests.
\subsection*{Funding} There is no funding for this paper.

\end{document}